\documentclass[a4paper]{article}
\usepackage[margin=1.2in]{geometry}
\usepackage{amsmath,amsthm}
\usepackage{amssymb}
\usepackage{amsfonts}
\usepackage{subfigure}
\usepackage{mathrsfs,bm}
\usepackage{array}
\usepackage{color,soul}
\usepackage{overpic}
\usepackage{ctable}
\usepackage{mathtools}
\usepackage{multirow}
\usepackage{authblk}
\usepackage{epstopdf}

\newtheorem{theorem}{Theorem}[section]
\newtheorem{lemma}[theorem]{Lemma}

\newcommand{\curl}{\mathrm{curl}}
\newcommand{\grad}{\mathrm{grad}}
\numberwithin{equation}{section}
\allowdisplaybreaks[4]

\begin{document}
\title{Superconvergence of a nonconforming brick element for the quad-curl problem\thanks{This project is supported by NNSFC (No.~12201254) and NSF of Jiangsu Province (No.~BK20200902).}}

\author[a]{Xinchen Zhou\thanks{Corresponding author: dasazxc@gmail.com}}
\author[b]{Zhaoliang Meng}
\author[a]{Hexin Niu}

\affil[a]{\it \small School of Mathematical Sciences, Jiangsu University, Zhenjiang, 212013, China}
\affil[b]{\it \small School of Mathematical Sciences, Dalian
University of Technology, Dalian, 116024, China}
\maketitle

\begin{abstract}
This short note shows the superconvergence of an $H(\grad\,\curl)$-nonconforming brick element very recently introduced in \cite{Wang2023} for the quad-curl problem.
The supercloseness is based on proper modifications for both the interpolation and the discrete formulation, leading to an $O(h^2)$ superclose order in the discrete $\grad\,\curl$ norm. Moreover, we propose a suitable postprocessing method to ensure the global superconvergence. Numerical results verify our theory.
\\[6pt]
\end{abstract}

\section{Introduction}

The quad-curl problem arises in various models such as the inverse electromagnetic scattering theory and magnetohydrodynamics. Recently, some 3D $H(\grad\,\curl)$-conforming finite elements were proposed \cite{Zhang2020, Hu2022, Wang2021}. However, they suffer from a complicated element structure, and nonconforming finite elements are therefore considered.
In \cite{Zheng2011}, Zheng et al.~constructed the first nonconforming tetrahedral element for the model problem. Huang \cite{Huang2023} also proposed two tetrahedral elements, where one coincides with the element in \cite{Zheng2011} and the other is a simplified version with fewer edge degrees of freedom (DoFs).
Very recently, Wang et al.~\cite{Wang2023} introduced two nonconforming cubical elements with similar structures as in Huang's contribution, but a remarkable fact is that they are robust for singular perturbation problems as well. All the above nonconforming elements are $H(\curl)$-nonconforming, and we also notice that some $H(\grad\,\curl)$-nonconforming but $H(\curl)$-conforming elements were presented, e.g., \cite{Zhang2022}.

To improve the computational precision,
superconvergence has been widely studied and applied for elements for various problems.
Comprehensive reviews can be found in \cite{Lin1996,Chen2002}.
In comparison with conforming elements, superconvergence analysis for nonconforming ones is
generally more difficult, in the sense that an extra consistency error term must be considered, and moreover, it is also a problem if the canonical interpolation itself is superclose.
Up to now, there have been many superconvergence results of nonconforming elements for the second order elliptic problem \cite{Lin2005,Hu2016}, for the biharmonic problem \cite{Hu2016, Hu2016b, Mao2009}, for the Stokes problem \cite{Liu2008, Ye2002}, etc.
Very recently, for the quad-curl problem, some superconvergence results were obtained in \cite{Wang2023, Zhang2023}, which both dealt with $\curl\,\curl$-conforming elements in 2D.
However, such a result for nonconforming elements seems invalid in the literature.

In this work, we consider the superconvergence of the simpler element proposed in \cite{Wang2023}, where only one DoF per edge is involved.
To overcome the difficulties mentioned above, two modifications are taken into account.
The first modification concerns the consistency error.
In \cite{Linke2014}, Linke discovered that the consistency error can be effectively controlled by a velocity reconstruction for nonconforming divergence-free Stokes elements, so that the velocity error can be independent of the pressure. This reconstruction only modifies the righthand side of the discrete formulation,
therefore the implementation is simple, and the computational complexity is nearly unchanged.
For the quad-curl problem and the element considered in this work, we adopt a similar idea using the Nédélec interpolation \cite{Nedelec1980}, which ensures an $O(h^2)$ order of the consistency error.
The second modification aims at the interpolation error. We modify the canonical interpolation of the element based on \cite{Zhou2023} for the superconvergence of the Stokes element designed in \cite{Zhang2009}. This technique was employed by Chen \cite{Chen2002} for nonconforming rectangular biharmonic elements, and thereafter utilized in \cite{Mao2009} and \cite{Hu2016b} for triangular and 3D elements. As a consequence, the modified interpolation is precisely superconvergent with order two again.
Finally, owing to the idea in \cite{Lin2000} for Maxwell's equations, we design a suitable postprocessing method using a macroelement interpolation to obtain the global superconvergence.
Numerical results confirm our theory.

The rest of this work is arranged as follows. Notations and preliminaries are given in Section \ref{s2}. Then the superclose and superconvergence results are demonstrated in Sections \ref{s3} and \ref{s4}, respectively. We also discuss the numerical experiments in Section \ref{s5}.

\section{Nonconforming brick element for the quad-curl problem}
\label{s2}
Let $\Omega\subset\mathbb{R}^3$ be a polyhedral domain and $(x_1,x_2,x_3)$ be the coordinate of a point $\bm{x}\in\Omega$.
Standard notations in Sobolev Spaces are used throughout this work.
For a domain $D\subset\Omega$,
the norms and seminorms in the Sobolev spaces $H^s(D)$ are indicated by $\|\cdot\|_{s,D}$ and $|\cdot|_{s,D}$, respectively.
The space $H_0^1(D)$ is the subspace of $H^1(D)$ with vanishing trace on $\partial D$.
In particular, we set $L^2(D)=H^0(D)$ with the inner product $(\cdot,\cdot)_D$.
These notations are also valid for vector- and matrix-valued Sobolev spaces, and the
subscript $D$ will be omitted if the domain $D=\Omega$.
We write $\bm{n}$ as the unit outward normal vector on $\partial D$,
and write $\bm{t}$ as a fixed unit tangent vector along $D$ if $D$ is one-dimensional.
The notation $P_k(D)$ denotes the usual polynomial space over $D$ of degree no more than $k$.
Moreover, $Q_{i,j,k}(D)$ represents the space of polynomials such that the maximum degree in variables $x_1,x_2$ and $x_3$ are $i,j,k$, respectively.
For simplicity we set $Q_k(D)=Q_{k,k,k}(D)$.
Furthermore, the constant $C>0$ independent of the mesh size may be taken different values in different places.

Some notations of vector-valued spaces are also useful. We set
\[
\begin{aligned}
&H^1(\curl;D)=\{\bm{v}\in[H^1(D)]^3:~\curl\,\bm{v}\in[H^1(D)]^3\},\\
&H(\grad\,\curl;D)=\{\bm{v}\in[L^2(D)]^3:~\curl\,\bm{v}\in[H^1(D)]^3\},\\
&H_0(\grad\,\curl;D)=\{\bm{v}\in H(\grad\,\curl;D):~\bm{v}\times\bm{n}=\curl\,\bm{v}=\bm{0}\mbox{~on~}\partial D\}.
\end{aligned}
\]

Assume that $\Omega$ can be divided into a set of cubical cells, denoted by $\mathcal{T}_h$.
To simplify our statement, we assume that the partition is uniform, but we point out that this restriction is not intrinsic, and main conclusions of this work are also available for general quasi-uniform meshes.
We select the coordinate system such that for any cell $K\in\mathcal{T}_h$, the edges of $K$ are parallel to the coordinate axes.
The diagram of $K$ can be expressed as $h=(h_1^2+h_2^2+h_3^2)^{1/2}$, where $h_i$ is the length of the edge of $K$ in the $x_i$-direction for each $i=1,2$ or $3$. We also set $\bm{x}_K$ as the center of $K$.
The sets of faces and edges of $K$ and in $\mathcal{T}_h$ are written as $\mathcal{F}_K$, $\mathcal{E}_K$, $\mathcal{F}_h$ and $\mathcal{E}_h$, respectively.

In \cite{Zhang2009}, Zhang et al.~introduced the following nonconforming brick element for the Brinkman model: For $K\in\mathcal{T}_h$, the shape function space $\bm{W}_K$ is defined by
\[
\bm{W}_K=[P_1(K)]^3\oplus\left(\mathrm{span}\{x_2^2,x_3^2\}\times\mathrm{span}\{x_3^2,x_1^2\}
\times\mathrm{span}\{x_1^2,x_2^2\}\right),
\]
and the DoFs are face integrals of $K$.
Utilizing this element, Wang et al.~\cite{Wang2023} proposed two $H(\grad\,\curl)$-nonconforming elements for the singularly perturbed quad-curl problem. The simpler one has the shape function space
\[
\bm{V}_K=\nabla\,Q_1(K)\oplus(\bm{x}-\bm{x}_K)\times\bm{W}_K,
\]
and the DoFs are
\begin{equation}
\label{e: WK dof}
\int_E\bm{v}\cdot\bm{t}\,\mathrm{d}s,~\forall E\in\mathcal{E}_K, ~\int_F\curl\,\bm{v}\cdot\bm{t}_{j,F}\,\mathrm{d}F,~j=1,2,~\forall F\in\mathcal{F}_K,
~\forall \bm{v}\in\bm{V}_K,
\end{equation}
where $\bm{t}_{j,F}$ are two unit vectors parallel to two non-colinear edges of $F$, respectively.
The homogeneous global spaces $\bm{W}_h$ and $\bm{V}_h$ with respect to the above two elements can be naturally defined in a piecewise manner by requiring that all the DoFs of cells in $\mathcal{T}_h$ are single-valued, and all the boundary DoFs vanish.
Moreover, we shall also use the $Q_1$ element space
\[
Q_h=\{q_h\in H_0^1(\Omega):~q_h|_K\in Q_1(K),~\forall K\in\mathcal{T}_h\}.
\]
According to \cite{Wang2023}, these spaces have the following relations:
\begin{equation}
\label{e: space relation}
\nabla Q_h\subset\bm{V}_h \mbox{ and } \curl_h\bm{V}_h\subset\bm{W}_h,
\end{equation}
where $\curl_h$ is the piecewise-defined version of the $\curl$ operator in the sense that $\curl_h|_K=\curl$, $\forall K\in\mathcal{T}_h$.

For $\bm{f}\in [L^2(\Omega)]^3$ and $\mathrm{div}\,\bm{f}=0$, the quad-curl problem has the form: Find $\bm{u}$ such that
\begin{equation}
\label{e: the quad-curl problem}
\begin{aligned}
\curl^4\,\bm{u}&=\bm{f}~~~~\mbox{in }\Omega,\\
\mathrm{div}\,\bm{u}&=0\,~~~~\mbox{in }\Omega,\\
\bm{u}\times\bm{n}=\curl\,\bm{u}\times\bm{n}&=\bm{0}~~~~\mbox{on }\partial\Omega.
\end{aligned}
\end{equation}
Due to the identity $\curl^2\,\bm{u}=-\Delta\bm{u}+\nabla\mathrm{div}\,\bm{u}$, the first equation is equivalent to
\begin{equation}
\label{e: curl L curl}
-\curl\,\Delta\,\curl\,\bm{u}=\bm{f}~~~~\mbox{in }\Omega.
\end{equation}
A variational formulation is to seek $(\bm{u},p)\in H_0(\grad\,\curl;\Omega)\times H_0^1(\Omega)$ such that
\begin{equation}
\label{e: variational formulation}
\begin{aligned}
a(\bm{u},\bm{v})+b(\bm{v},p)&=(\bm{f},\bm{v}),~\forall\bm{v}\in H_0(\grad\,\curl;\Omega),\\
b(\bm{u},q)&=0,~\forall q\in H_0^1(\Omega),
\end{aligned}
\end{equation}
where $a(\bm{u},\bm{v})=(\nabla\curl\,\bm{u},\nabla\curl\,\bm{v})$
and $b(\bm{v},p)=(\bm{v},\nabla p)$. The standard discrete formulation \cite{Wang2023,Huang2023} can be written as: Find $(\bm{u}_h^0,p_h^0)\in\bm{V}_h\times Q_h$ fulfilling
\begin{equation}
\label{e: discrete formulation}
\begin{aligned}
a_h(\bm{u}_h^0,\bm{v}_h)+b(\bm{v}_h,p_h^0)&=(\bm{f},\bm{v}_h),~\forall\bm{v}\in \bm{V}_h,\\
b(\bm{u}_h^0,q_h)&=0,~\forall q_h\in Q_h,
\end{aligned}
\end{equation}
where $a_h(\bm{u}_h,\bm{v}_h)=(\nabla_h\curl_h\bm{u}_h,\nabla_h\curl_h\bm{v}_h)$ and $\nabla_h$ is
the piecewise-defined gradient operator.
Since $\mathrm{div}\,\bm{f}=0$ and (\ref{e: space relation}) holds, by taking $\bm{v}=\nabla p$ in the first equation in (\ref{e: variational formulation}) and taking $\bm{v}_h=\nabla p_h^0$ in the first equation in (\ref{e: discrete formulation}), one must have $p=p_h^0=0$. Furthermore, for piecewise sufficiently smooth functions $\bm{v}$, define the discrete norm or semi-norm
\[
|\bm{v}|_{1,h}^2=\sum_{K\in\mathcal{T}_h}|\bm{v}|_{1,K}^2,~
\|\bm{v}\|_{\grad\,\curl,h}^2=\|\bm{v}\|_0^2+\|\curl_h\bm{v}\|_0^2+|\curl_h\bm{v}|_{1,h}^2.
\]
Then by \cite{Wang2023,Huang2023}, if $\bm{u}\in [H^1(\Omega)]^3$ and $\curl\,\bm{u}\in[H^2(\Omega)]^3$, then the following a priori error estimate holds:
\begin{equation}
\label{e: orginal err}
\|\bm{u}-\bm{u}_h^0\|_{\grad\,\curl,h}\leq Ch(\|\curl\,\bm{u}\|_2+|\bm{u}|_1+\|\bm{f}\|_0).
\end{equation}

\section{Supercloseness}
\label{s3}
In comparison with conforming elements, superconvergence analysis for nonconforming ones is generally more difficult, in the sense that an extra consistency
error term must be carefully dealt with, and moreover, the canonical interpolation is often not superclose.
To overcome the first difficulty, inspired by \cite{Linke2014}, we shall modify the righthand side of (\ref{e: discrete formulation}).
To this end, recall the well-known Nédélec brick element \cite{Nedelec1980} of the lowest order, whose local shape function space is
\[
\bm{V}^C_K=Q_{0,1,1}(K)\times Q_{1,0,1}(K)\times Q_{1,1,0}(K),
\]
and the local DoFs are precisely the edge DoFs in (\ref{e: WK dof}). We also write the canonical interpolation operator preserving the above DoFs as $\bm{I}^C_K$. Let $\bm{V}^C_h$ be the corresponding homogeneous global spaces and let $\bm{I}^C_h$ be the global version of $\bm{I}^C_K$ such that $\bm{I}^C_h|_K=\bm{I}^C_K$, then the modified formulation reads: Find $(\bm{u}_h,p_h)\in\bm{V}_h\times Q_h$ such that
\begin{equation}
\label{e: modified discrete formulation}
\begin{aligned}
a_h(\bm{u}_h,\bm{v}_h)+b(\bm{v}_h,p_h)&=(\bm{f},\bm{I}^C_h\bm{v}_h),~\forall\bm{v}\in \bm{V}_h,\\
b(\bm{u}_h,q_h)&=0,~\forall q_h\in Q_h.
\end{aligned}
\end{equation}
The finite element matrices of (\ref{e: discrete formulation}) and (\ref{e: modified discrete formulation}) are precisely the same, and therefore the wellposedness of (\ref{e: modified discrete formulation}) is a direct consequence of that of (\ref{e: discrete formulation}).
Moreover, since $\nabla Q_h\subset\bm{V}_h^C$ (see e.g.~\cite{Monk2003}), taking $\bm{v}_h=\nabla p_h$ in the first equation in (\ref{e: modified discrete formulation}) we get
\[
\|\nabla p_h\|_0^2=(\bm{f},\bm{I}^C_h\nabla p_h)=(\bm{f},\nabla p_h)=0,
\]
and thus again we arrive at $p_h=0$.

To overcome the second difficulty, we modify the canonical interpolations with respect to $\bm{W}_h$ and $\bm{V}_h$. For $K\in\mathcal{T}_h$, we design the local interpolation $\bm{\Pi}_K:~[H^3(K)]^3\rightarrow\bm{W}_K$ by setting
\[
\begin{aligned}
&\int_F\bm{\Pi}_K\bm{w}\cdot\bm{t}_{j,F}\,\mathrm{d}F
=\int_F\left(\bm{w}\cdot\bm{t}_{j,F}+\frac{h_{k_{j,F}}^2}{12}\frac{\partial^2(\bm{w}\cdot\bm{t}_{j,F})}{\partial x_{k_{j,F}}^2}\right)\,\mathrm{d}F,~j=1,2,~\forall F\in\mathcal{F}_K,\\
&\int_F\bm{\Pi}_K\bm{w}\cdot\bm{n}_F\,\mathrm{d}F=\int_F\bm{w}\cdot\bm{n}_F\,\mathrm{d}F,
~\forall F\in\mathcal{F}_K,
\end{aligned}
\]
where $\bm{n}_F$ is a fixed unit normal of $F$, and $k_{j,F}$ is taken $1$, $2$ or $3$ such that $\bm{t}_{j,F}$ is parallel to the $x_{k_{j,F}}$-axis. For example, if $\bm{t}_{j,F}=(1,0,0)^T$, then $\bm{t}_{j,F}$ is parallel to the $x_1$-axis, and therefore $k_{j,F}=1$.
This operator also works for the superconvergence of $\bm{W}_h$ for the Stokes problem with piecewise constant pressure \cite{Zhou2023}.
Similarly, if $\bm{v}\in [H^1(K)]^3$ and $\curl\,\bm{v}\in [H^3(K)]^3$, then the local interpolation $\bm{I}_K\bm{v}\in \bm{V}_K$ can be determined by
\[
\begin{aligned}
&\int_F\curl\,\bm{I}_K\bm{v}\cdot\bm{t}_{j,F}\,\mathrm{d}F
=\int_F\left(\curl\,\bm{v}\cdot\bm{t}_{j,F}+\frac{h_{k_{j,F}}^2}{12}\frac{\partial^2(\curl\,\bm{v}\cdot\bm{t}_{j,F})}{\partial x_{k_{j,F}}^2}\right)\,\mathrm{d}F,~j=1,2,~\forall F\in\mathcal{F}_K,\\
&\int_E\bm{I}_K\bm{v}\cdot\bm{t}\,\mathrm{d}s=\int_E\bm{v}\cdot\bm{t}\,\mathrm{d}s,
~\forall E\in\mathcal{E}_K.
\end{aligned}
\]
The global interpolations $\bm{\Pi}_h$ and $\bm{I}_h$ are naturally piecewise defined: $\bm{\Pi}_h|_K=\bm{\Pi}_K$ and $\bm{I}_h|_K=\bm{I}_K$, $\forall K\in\mathcal{T}_h$. We must point out that these interpolations are well-defined and continuous with the claimed regularity due to the trace theorem and Lemma 3.1 in \cite{Wang2023}. In addition, if $\bm{w}\in [H_0^1(\Omega)]^3$ and $\bm{v}\in H_0(\grad\,\curl;\Omega)$, then
$\bm{\Pi}_h\bm{w}\in\bm{W}_h$ and $\bm{I}_h\bm{v}\in\bm{V}_h$ due to the vanishing trace of $\bm{w}$ or $\curl\,\bm{v}$ over $\partial\Omega$.

The supercloseness is based on the following five lemmas.

\begin{lemma}
\label{lemma: WK interp superclose}
For all $\bm{w}\in [H^3(K)]^3$ and $K\in\mathcal{T}_h$, one has
\[
|(\nabla(\bm{w}-\bm{\Pi}_K\bm{w}),\nabla\bm{w}_h)_K|\leq Ch^2|\bm{w}|_{3,K}|\bm{w}_h|_{1,K},
~\forall \bm{w}_h\in\bm{W}_K.
\]
\end{lemma}
\begin{proof}
Since we have assumed that the mesh is uniform, let $\widehat{K}=\{\widehat{\bm{x}}=(\bm{x}-\bm{x}_K)/h, \bm{x}\in K\}$ be the reference cell independent of $K$
and $\widehat{\bm{w}}(\widehat{\bm{x}})=\bm{w}(\bm{x})$ for any function $\bm{w}$ defined over $K$.
Then it can be checked by a direct computation that
\[
\begin{aligned}
&(\nabla(\widehat{\bm{w}}-\bm{\Pi}_{\widehat{K}}\widehat{\bm{w}}),\nabla\widehat{\bm{w}}_h)_{\widehat{K}}=0,
~\forall\widehat{\bm{w}}\in [P_2(\widehat{K})]^3,~\forall\widehat{\bm{w}}_h\in\bm{W}_{\widehat{K}},\\
&|(\nabla(\widehat{\bm{w}}-\bm{\Pi}_{\widehat{K}}\widehat{\bm{w}}),\nabla\widehat{\bm{w}}_h)_{\widehat{K}}|\leq C\|\widehat{\bm{w}}\|_{3,\widehat{K}}|\widehat{\bm{w}}_h|_{1,\widehat{K}},~\forall\widehat{\bm{w}}\in[H^3(\widehat{K})]^3, ~\forall\widehat{\bm{w}}_h\in\bm{W}_{\widehat{K}}.
\end{aligned}
\]
According to the definition of $\bm{W}_K$, one has $\bm{W}_{\widehat{K}}=\mathrm{span}\{\widehat{\bm{w}},~\bm{w}\in\bm{W}_K\}$.
Moreover, by checking the DoFs of $\widehat{\bm{\Pi}_{K}\bm{w}}$ and $\bm{\Pi}_{\widehat{K}}\widehat{\bm{w}}$ we find $\widehat{\bm{\Pi}_{K}\bm{w}}=\bm{\Pi}_{\widehat{K}}\widehat{\bm{w}}$.
Therefore a scaling argument and the Bramble--Hilbert lemma show that
\[
\begin{aligned}
|(\nabla(\bm{w}-\bm{\Pi}_K\bm{w}),\nabla\bm{w}_h)_K|
&\leq Ch\inf_{\widehat{\bm{q}}\in [P_2(\widehat{K})]^3}|(\nabla((\widehat{\bm{w}}-\widehat{\bm{q}})-\bm{\Pi}_{\widehat{K}}(\widehat{\bm{w}}-\widehat{\bm{q}})),\nabla \widehat{\bm{w}}_h)_{\widehat{K}}|\\
&\leq Ch\inf_{\widehat{\bm{q}}\in [P_2(\widehat{K})]^3}
\|\bm{\widehat{w}}-\widehat{\bm{q}}\|_{3,\widehat{K}}|\widehat{\bm{w}}_h|_{1,\widehat{K}}\\
&\leq Ch|\bm{\widehat{w}}|_{3,\widehat{K}}|\widehat{\bm{w}}_h|_{1,\widehat{K}}
= Ch^2|\bm{w}|_{3,K}|\bm{w}_h|_{1,K},
\end{aligned}
\]
and the proof is done.
\end{proof}

\begin{lemma}
\label{lemma: bh superclose}
For all $\bm{v}$ such that $\bm{v}\in [H^2(K)]^3$ and $\curl\,\bm{v}\in[H^3(K)]^3$, it holds
\[
|(\bm{v}-\bm{I}_K\bm{v},\nabla q_h)_K|\leq Ch^2(|\bm{v}|_{2,K}+h|\curl\,\bm{v}|_{2,K}+h^2|\curl\,\bm{v}|_{3,K})|q_h|_{1,K},~\forall q_h\in Q_1(K), ~\forall K\in\mathcal{T}_h.
\]
\end{lemma}
\begin{proof}
Denote by $\bm{I}_K^0$ the canonical interpolation operator with respect to $\bm{V}_K$, namely all the DoFs in (\ref{e: WK dof}) are strictly preserved.
Then over the reference cell $\widehat{K}$, by a direct computation and the continuity of $\bm{I}_{\widehat{K}}^0$ (Lemma 3.1 in \cite{Wang2023}), one obtains for any $\widehat{q}_h\in Q_1(\widehat{K})$,
\[
\begin{aligned}
(\widehat{\bm{v}}-\bm{I}_{\widehat{K}}^0\widehat{\bm{v}},\nabla \widehat{q}_h)_{\widehat{K}}&=0,
~\forall\widehat{\bm{v}}\in [P_1(\widehat{K})]^3,\\
|(\widehat{\bm{v}}-\bm{I}_{\widehat{K}}^0\widehat{\bm{v}},\nabla \widehat{q}_h)_{\widehat{K}}|&\leq C(\|\widehat{\bm{v}}\|_{1,\widehat{K}}+\|\curl\,\widehat{\bm{v}}\|_{1,\widehat{K}})|\widehat{q}_h|_{1,\widehat{K}},
~\forall\widehat{\bm{v}}\in H^1(\curl;\widehat{K}),
\end{aligned}
\]
which gives in a similar fashion as in the proof of Lemma \ref{lemma: WK interp superclose} that
\begin{equation}
\label{e: IK0 superclose}
|(\bm{v}-\bm{I}_K^0\bm{v},\nabla q_h)_K|\leq Ch^2|\bm{v}|_{2,K}|q_h|_{1,K},
~\forall\bm{v}\in [H^2(K)]^3,~\forall q_h\in Q_1(K).
\end{equation}
On the other hand, by the definition of $\bm{I}_K$, a scaling argument and the trace theorem,
one has
\begin{equation}
\label{e: modified interp err}
\begin{aligned}
|(\bm{I}_K^0\bm{v}-\bm{I}_K\bm{v},\nabla q_h)_K|
&\leq Ch^2\sum_{\widehat{F}\in\mathcal{F}_{\widehat{K}}}\sum_{j=1}^2
\left|\int_{\widehat{F}}\frac{\partial^2(\curl\,\widehat{\bm{v}}\cdot\bm{t}_{j,\widehat{F}})}{\partial \widehat{x}_{k_{j,\widehat{F}}}^2}\,\mathrm{d}\widehat{F}\right||\widehat{q}_h|_{1,\widehat{K}}\\
&\leq Ch^2(|\curl\,\widehat{\bm{v}}|_{2,\widehat{K}}+|\curl\,\widehat{\bm{v}}|_{3,\widehat{K}})|\widehat{q}_h|_{1,\widehat{K}}\\
&\leq Ch^3(|\curl\,\bm{v}|_{2,K}+h|\curl\,\bm{v}|_{3,K})|q_h|_{1,K}.
\end{aligned}
\end{equation}
A summation of (\ref{e: IK0 superclose}) and (\ref{e: modified interp err}) completes the proof.
\end{proof}

\begin{lemma}
\label{lemma: commute}
The following commutative diagram holds: if $\bm{v}\in [H^1(K)]^3$ and $\curl\,\bm{v}\in [H^3(K)]^3$, then
\[
\bm{\Pi}_K\curl\,\bm{v}=\curl\,\bm{I}_K\bm{v},~\forall K\in\mathcal{T}_h.
\]
\end{lemma}
\begin{proof}
By (\ref{e: space relation}) we have $\curl\,\bm{V}_K\subset\bm{W}_K$,
therefore it suffices to show that all the DoFs of $\bm{\Pi}_K\curl\,\bm{v}$ and $\curl\,\bm{I}_K\bm{v}$ with respect to $\bm{W}_K$ are the same.
This is indeed a direct consequence of the Stokes formula (for the normal part) and the definitions
of $\bm{\Pi}_K$ and $\bm{I}_K$ (for the tangent part).
\end{proof}

\begin{lemma}
\label{lemma: consistency Wh}
For any $\bm{\phi}\in [H^2(\Omega)]^3$ and $\bm{w}_h\in\bm{W}_h$,
one has
\[
\sum_{K\in\mathcal{T}_h}\int_{\partial K}\bm{\phi}\cdot\bm{w}_h\,\mathrm{d}F\leq Ch^2|\bm{\phi}|_2|\bm{w}_h|_{1,h}.
\]
\end{lemma}
\begin{proof}
It is easy to check that functions in $\bm{W}_h$ have the following two properties:
(i) any face integral of the jump of each component vanishes;
(ii) each component restricted to $\bm{W}_K$ can be written as the sum of three univariate functions with respect to $x_1$, $x_2$ or $x_3$. Therefore by the spirit of Lemma 3.1 in \cite{Lin2005}, the conclusion is established.
\end{proof}

\begin{lemma}
\label{lemma: consistency Vh}
For all $\bm{v}_h\in\bm{V}_K$ and $K\in\mathcal{T}_h$, the following equation holds:
\[
(\bm{q},\curl(\bm{v}_h-\bm{I}_K^C\bm{v}_h))_K=0,~\forall \bm{q}\in [P_0(K)]^3.
\]
\end{lemma}
\begin{proof} Lemma 5.2 in \cite{Wang2023} asserts
\[
\int_{\partial K}\bm{q}\cdot(\bm{v}_h-\bm{I}_K^C\bm{v}_h)\times\bm{n}\,\mathrm{d}F=0,~\forall \bm{q}\in [P_0(K)]^3,
\]
then the conclusion follows from an integration by parts.
\end{proof}

Now we are in a position to provide the superclose result.

\begin{theorem}
\label{th: supercloseness}
Assume that the exact solution $\bm{u}$ of problem (\ref{e: variational formulation}) has the regularity
$\bm{u}\in [H^2(\Omega)]^3$ and $\curl\,\bm{u}\in[H^3(\Omega)]^3$.
Then the discrete solution $\bm{u}_h$ of problem (\ref{e: modified discrete formulation}) satisfies the following superclose property:
\[
\|\bm{I}_h\bm{u}-\bm{u}_h\|_{\grad\,\curl,h}\leq Ch^2(|\bm{u}|_2+\|\curl\,\bm{u}\|_3).
\]
\end{theorem}

\begin{proof}
By the standard theory of mixed elements \cite{Boffi2013}, the wellposedness of (\ref{e: modified discrete formulation}) (or equivalently (\ref{e: discrete formulation})) implies the following
discrete stability:
\[
\|\widetilde{\bm{v}}_h\|_{\grad\,\curl,h}+|\widetilde{q}_h|_1
\leq C\sup_{(\bm{v}_h,q_h)\in\bm{V}_h\times Q_h}
\frac{a_h(\widetilde{\bm{v}}_h,\bm{v}_h)+b(\bm{v}_h,\widetilde{q}_h)+b(\widetilde{\bm{v}}_h,q_h)}
{\|\bm{v}_h\|_{\grad\,\curl,h}+|q_h|_1},
~\forall (\bm{\widetilde{\bm{v}}_h},\widetilde{q}_h)\in\bm{V}_h\times Q_h.
\]
See also Lemma 4.4 in \cite{Huang2023}.
Taking $\widetilde{\bm{v}}_h=\bm{I}_h\bm{u}-\bm{u}_h$, $\widetilde{q}_h=0$, and combining (\ref{e: variational formulation}) with (\ref{e: modified discrete formulation}), one then has
\begin{equation}
\label{e: superclose all}
\begin{aligned}
&\|\bm{I}_h\bm{u}-\bm{u}_h\|_{\grad\,\curl,h}\\
\leq& C\sup_{(\bm{v}_h,q_h)\in\bm{V}_h\times Q_h}
\frac{a_h(\bm{I}_h\bm{u}-\bm{u}_h,\bm{v}_h)+b(\bm{I}_h\bm{u}-\bm{u}_h,q_h)}
{\|\bm{v}_h\|_{\grad\,\curl,h}+|q_h|_1}\\
\leq& C\sup_{(\bm{v}_h,q_h)\in\bm{V}_h\times Q_h}
\frac{a_h(\bm{I}_h\bm{u},\bm{v}_h)-(\bm{f},\bm{I}_h^C\bm{v}_h)+b(\bm{I}_h\bm{u}-\bm{u},q_h)+b(\bm{u}-\bm{u}_h,q_h)}
{\|\bm{v}_h\|_{\grad\,\curl,h}+|q_h|_1}\\
\leq& C\sup_{(\bm{v}_h,q_h)\in\bm{V}_h\times Q_h}
\frac{a_h(\bm{I}_h\bm{u}-\bm{u},\bm{v}_h)+a(\bm{u},\bm{v}_h)-(\bm{f},\bm{I}_h^C\bm{v}_h)+b(\bm{I}_h\bm{u}-\bm{u},q_h)}
{\|\bm{v}_h\|_{\grad\,\curl,h}+|q_h|_1}.
\end{aligned}
\end{equation}
For the first term, it follows from (\ref{e: space relation}) that $\curl_h\bm{v}_h\in\bm{W}_h$. Then by Lemmas \ref{lemma: WK interp superclose} and \ref{lemma: commute},
\begin{equation}
\label{e: superclose first}
\begin{aligned}
|a_h(\bm{I}_h\bm{u}-\bm{u},\bm{v}_h)|
&=|(\nabla_h(\curl\,\bm{u}-\curl_h\bm{I}_h\bm{u}),\nabla_h\curl\,\bm{v}_h)|\\
&=|(\nabla_h(\curl\,\bm{u}-\bm{\Pi}_h\curl\,\bm{u}),\nabla_h\curl\,\bm{v}_h)|\\
&\leq Ch^2|\curl\,\bm{u}|_3|\curl\,\bm{v}_h|_{1,h}.
\end{aligned}
\end{equation}
For the last term, we adopt Lemma \ref{lemma: bh superclose} to derive
\begin{equation}
\label{e: superclose last}
|b(\bm{I}_h\bm{u}-\bm{u},q_h)|\leq Ch^2(|\bm{u}|_2+h|\curl\,\bm{u}|_2+h^2|\curl\,\bm{u}|_3)|q_h|_1
\leq Ch^2(|\bm{u}|_2+\|\curl\,\bm{u}\|_3)|q_h|_1.
\end{equation}
It suffices to check the consistency error term. To this end, integrating by parts using (\ref{e: the quad-curl problem}) and (\ref{e: curl L curl}) one sees
\begin{equation}
\label{e: superclose consistency}
\begin{aligned}
a(\bm{u},\bm{v}_h)-(\bm{f},\bm{I}_h^C\bm{v}_h)
&=a(\bm{u},\bm{v}_h)+(\Delta\curl\,\bm{u},\curl_h\bm{v}_h)-(\Delta\curl\,\bm{u},\curl_h\bm{v}_h)-(\bm{f},\bm{I}_h^C\bm{v}_h)\\
&=\sum_{K\in\mathcal{T}_h}\int_{\partial K}\frac{\partial(\curl\,\bm{u})}{\partial\bm{n}}\cdot\curl_h\bm{v}_h\,\mathrm{d}F
-(\Delta\curl\,\bm{u},\curl_h(\bm{v}_h-\bm{I}_h^C\bm{v}_h)).
\end{aligned}
\end{equation}
By taking $\bm{\phi}=\partial(\curl\,\bm{u})/\partial\bm{n}$ in Lemma \ref{lemma: consistency Wh} we obtain
\begin{equation}
\label{e: consistency 1}
\left|\sum_{K\in\mathcal{T}_h}\int_{\partial K}\frac{\partial(\curl\,\bm{u})}{\partial\bm{n}}\cdot\curl_h\bm{v}_h\,\mathrm{d}F\right|\leq Ch^2|\curl\,\bm{u}|_3|\curl_h\bm{v}_h|_{1,h},
\end{equation}
Moreover, let $\mathcal{P}_{0,K}$ be the $L^2$-projection operator to $[P_0(K)]^3$. Lemma \ref{lemma: consistency Vh} and the standard interpolation error estimate for the Nédélec element (Theorem 6.6 in \cite{Monk2003}) imply that
\begin{equation}
\label{e: consistency 2}
\begin{aligned}
|(\Delta\curl\,\bm{u},\curl_h(\bm{v}_h-\bm{I}_h^C\bm{v}_h))|
&=\left|\sum_{K\in\mathcal{T}_h}(\Delta\curl\,\bm{u}-\mathcal{P}_{0,K}\Delta\curl\,\bm{u},\curl_h(\bm{v}_h-\bm{I}_h^C\bm{v}_h))_K\right|\\
&\leq Ch^2|\curl\,\bm{u}|_3|\curl_h\bm{v}_h|_{1,h}.
\end{aligned}
\end{equation}
Substituting (\ref{e: superclose first})--(\ref{e: consistency 2}) into (\ref{e: superclose all}) results in the desired estimate.
\end{proof}

\section{Postprocessing}
\label{s4}
We follow the idea in \cite{Lin2000} to design a postprocessing method to achieve the global superconvergence.
Let $M$ be a brick macroelement consisting of $3\times3\times3$ adjacent cells in $\mathcal{T}_h$
and introduce the sets $\mathcal{F}_M=\bigcup_{K\subset M}\mathcal{F}_K$ and $\mathcal{E}_M=\bigcup_{K\subset M}\mathcal{E}_K$.
We construct two elements over $M$: for the first element, we select the shape function space $\bm{V}_M$ as the second order Nédélec brick element space, namely,
\[
\bm{V}_M=Q_{2,3,3}(M)\times Q_{3,2,3}(M)\times Q_{3,3,2}(M),
\]
and the DoFs are
\[
\int_E\bm{v}\cdot\bm{t}\,\mathrm{d}s,~\forall E\in\mathcal{E}_M,~\forall \bm{v}\in\bm{V}_M.
\]
The second element has the shape function space $\bm{W}_M$:
\[
\bm{W}_M=Q_{3,2,2}(M)\times Q_{2,3,2}(M)\times Q_{2,2,3}(M),
\]
which is precisely the same as the second order Raviart--Thomas--Nédélec brick element space \cite{Nedelec1980}, and the associated DoFs are
\[
\int_F\bm{w}\cdot\bm{n}_F\,\mathrm{d}F,~\forall F\in\mathcal{F}_M,~\forall \bm{w}\in\bm{W}_M.
\]
It is easy to see these two elements are well-defined.
Moreover, let $\bm{I}_M$ and $\bm{\Pi}_M$ be their canonical interpolation operators, then
as in Lemma \ref{lemma: commute}, the following commutative diagram holds:
\begin{equation}
\label{e: post commute}
\bm{\Pi}_M\curl_h\bm{v}=\curl_h\bm{I}_M\bm{v},~\forall\bm{v}\in H^1(\curl; M)+\bm{V}_h|_M.
\end{equation}
As a consequence, for a sufficiently smooth function $\bm{v}$,
we have the interpolation error estimate
\begin{equation}
\label{e: macro interp err}
\begin{aligned}
&\|\bm{v}-\bm{I}_M\bm{v}\|_{0,M}\leq Ch^2|\bm{v}|_{2,M},\\
&|\curl(\bm{v}-\bm{I}_M\bm{v})|_{j,M}=|\curl\,\bm{v}-\bm{\Pi}_M\curl\,\bm{v}|_{j,M}
\leq Ch^{3-j}|\curl\,\bm{v}|_{3,M},~j=0,1.
\end{aligned}
\end{equation}
Let $\bm{I}_{3h}$ be the global version of $\bm{I}_M$, then it follows from the definitions of $\bm{I}_M$, $\bm{I}_h$ and (\ref{e: post commute}), (\ref{e: macro interp err}) that for all sufficiently smooth $\bm{v}$ over $\Omega$ and $\bm{v}_h\in\bm{V}_h$,
\begin{equation}
\label{e: I3h}
\begin{aligned}
\bm{I}_{3h}\bm{I}_h\bm{v}&=\bm{I}_{3h}\bm{v},~\|\bm{v}-\bm{I}_{3h}\bm{v}\|_{\grad\,\curl,h}\leq Ch^2(|\bm{v}|_2+|\curl\,\bm{v}|_3),\\
\|\bm{I}_{3h}\bm{v}_h\|_{\grad\,\curl,h}^2
&=\sum_M\left(\|\bm{I}_M\bm{v}_h\|_{0,M}^2+\|\bm{\Pi}_M\curl_h\bm{v}_h\|_{0,M}^2+|\bm{\Pi}_M\curl_h\bm{v}_h|_{1,M}^2\right)\\
&\leq C\sum_M\sum_{K\subset M} \left(\|\bm{v}_h\|_{0,K}^2+\|\curl_h\bm{v}_h\|_{0,K}^2+|\curl_h\bm{v}_h|_{1,K}^2\right)
=C\|\bm{v}_h\|_{\grad\,\curl,h}^2.
\end{aligned}
\end{equation}
Then we can state our superconvergence result.

\begin{theorem}
\label{th: superconvergence}
Assume that the exact solution $\bm{u}$ of problem (\ref{e: variational formulation}) has the regularity
$\bm{u}\in [H^2(\Omega)]^3$ and $\curl\,\bm{u}\in[H^3(\Omega)]^3$.
Then for the discrete solution $\bm{u}_h$ of problem (\ref{e: modified discrete formulation}),
the postprocessed solution $\bm{I}_{3h}\bm{u}_h$ achieves the following superconvergence:
\[
\|\bm{u}-\bm{I}_{3h}\bm{u}_h\|_{\grad\,\curl,h}\leq Ch^2(|\bm{u}|_2+\|\curl\,\bm{u}\|_3).
\]
\end{theorem}
\begin{proof}
According to Theorem \ref{th: supercloseness} and (\ref{e: I3h}),
\[
\begin{aligned}
\|\bm{u}-\bm{I}_{3h}\bm{u}_h\|_{\grad\,\curl,h}
&\leq\|\bm{u}-\bm{I}_{3h}\bm{I}_h\bm{u}\|_{\grad\,\curl,h}
+\|\bm{I}_{3h}(\bm{I}_h\bm{u}-\bm{u}_h)\|_{\grad\,\curl,h}\\
&\leq\|\bm{u}-\bm{I}_{3h}\bm{u}\|_{\grad\,\curl,h}+C\|\bm{I}_h\bm{u}-\bm{u}_h\|_{\grad\,\curl,h}\\
&\leq Ch^2(|\bm{u}|_2+\|\curl\,\bm{u}\|_3),
\end{aligned}
\]
which completes the proof.
\end{proof}

\section{Numerical results}
\label{s5}
In this section, we report and analyze the results of a numerical example.
Let the solution domain be $\Omega=[0,1]^3$.
We employ the same quad-curl problem as in \cite{Huang2023}, where the exact solution of $(\ref{e: the quad-curl problem})$ is given by
\[
\bm{u}=\curl\left(0,0,\sin^3(\pi x)\sin^3(\pi y)\sin^3(\pi z)\right)^T,
\]
and the righthand-side term $\bm{f}$ can be accordingly calculated.
The meshes are generated by $n\times n\times n$ uniform partitions.
We first investigate the performance of the original scheme (\ref{e: discrete formulation})
and the modified counterpart (\ref{e: modified discrete formulation}).
These results are given in Tables \ref{t: original} and \ref{t: modified}.
The supercloseness and superconvergence of the scheme (\ref{e: modified discrete formulation}) are checked with the results listed in Tables \ref{t: supercloseness} and \ref{t: superconvergence}.

\begin{table}[!htb]
\begin{center}
\begin{tabular}{p{0.5cm}<{\centering}p{2.8cm}<{\centering}p{0.8cm}<{\centering}p{2.8cm}<{\centering}p{0.8cm}
<{\centering}p{2.8cm}<{\centering}p{0.8cm}<{\centering}}
\toprule
$n$ & $|\curl_h(\bm{u}-\bm{u}_h^0)|_{1,h}$ & order & $\|\curl_h(\bm{u}-\bm{u}_h^0)\|_0$  & order & $\|\bm{u}-\bm{u}_h^0\|_0$ & order \\
\midrule
$6$      &4.351E1     &       &1.548E0      &         &2.244E-1    &         \\
$12$     &2.166E1     &1.01   &4.096E-1     &1.92     &1.076E-1    &1.06     \\
$18$     &1.441E1     &1.01   &1.841E-1     &1.97     &7.105E-2    &1.03     \\
$24$     &1.080E1     &1.00   &1.039E-1     &1.99     &5.307E-2    &1.01     \\
$36$     &7.198E0     &1.00   &4.635E-2     &1.99     &3.528E-2    &1.01     \\
$48$     &5.397E0     &1.00   &2.610E-2     &2.00     &2.643E-2    &1.00     \\
\bottomrule
\end{tabular}
\caption{Numerical results of the original scheme (\ref{e: discrete formulation}) for different $n$.\label{t: original}}
\end{center}
\end{table}

\begin{table}[!htb]
\begin{center}
\begin{tabular}{p{0.5cm}<{\centering}p{2.8cm}<{\centering}p{0.8cm}<{\centering}p{2.8cm}<{\centering}p{0.8cm}
<{\centering}p{2.8cm}<{\centering}p{0.8cm}<{\centering}}
\toprule
$n$ & $|\curl_h(\bm{u}-\bm{u}_h)|_{1,h}$ & order & $\|\curl_h(\bm{u}-\bm{u}_h)\|_0$  & order & $\|\bm{u}-\bm{u}_h\|_0$ & order \\
\midrule
$6$      &4.332E1     &       &1.836E0      &         &2.344E-1    &         \\
$12$     &2.159E1     &1.00   &4.734E-1     &1.96     &1.081E-1    &1.12     \\
$18$     &1.439E1     &1.00   &2.118E-1     &1.98     &7.112E-2    &1.03     \\
$24$     &1.079E1     &1.00   &1.194E-1     &1.99     &5.309E-2    &1.02     \\
$36$     &7.194E0     &1.00   &5.317E-2     &2.00     &3.528E-2    &1.01     \\
$48$     &5.395E0     &1.00   &2.992E-2     &2.00     &2.643E-2    &1.00     \\
\bottomrule
\end{tabular}
\caption{Numerical results of the modified scheme (\ref{e: modified discrete formulation}) for different $n$.\label{t: modified}}
\end{center}
\end{table}

One can observe from Tables \ref{t: original} and \ref{t: modified} that the errors for both schemes measured in $|\curl_h\cdot|_{1,h}$, $\|\curl_h\cdot\|_0$ and $\|\cdot\|_0$ are all at least first order convergent,
which is consistent with (\ref{e: orginal err}). Moreover, as first discovered in \cite{Huang2023}, the convergence order in $\|\curl_h\cdot\|_0$ is indeed two.
We also remark that although the absolute errors of these two schemes are almost equal, but for the original one we do not know if the supercloseness can be guaranteed in theory.

\begin{table}[!htb]
\begin{center}
\begin{tabular}{p{0.5cm}<{\centering}p{3cm}<{\centering}p{0.8cm}<{\centering}p{3cm}<{\centering}p{0.8cm}
<{\centering}p{2.8cm}<{\centering}p{0.8cm}<{\centering}}
\toprule
$n$ & $|\curl_h(\bm{I}_h\bm{u}-\bm{u}_h)|_{1,h}$ & order & $\|\curl_h(\bm{I}_h\bm{u}-\bm{u}_h)\|_0$  & order & $\|\bm{I}_h\bm{u}-\bm{u}_h\|_0$ & order \\
\midrule
$6$      &1.092E1      &       &8.394E-1     &         &8.565E-2    &         \\
$12$     &2.577E0      &2.08   &2.092E-1     &2.00     &2.242E-2    &1.93     \\
$18$     &1.125E0      &2.04   &9.300E-2     &2.00     &1.000E-2    &1.99     \\
$24$     &6.284E-1     &2,02   &5.231E-2     &2.00     &5.636E-3    &2.00     \\
$36$     &2.778E-1     &2.01   &2.325E-2     &2.00     &2.507E-3    &2.00     \\
$48$     &1.559E-1     &2.01   &1.308E-2     &2.00     &1.410E-3    &2.00     \\
\bottomrule
\end{tabular}
\caption{Supercloseness of the modified scheme (\ref{e: modified discrete formulation}) using the interpolation operator $\bm{I}_h$.\label{t: supercloseness}}
\end{center}
\end{table}

\begin{table}[!htb]
\begin{center}
\begin{tabular}{p{0.5cm}<{\centering}p{3.1cm}<{\centering}p{0.8cm}<{\centering}p{3cm}<{\centering}p{0.8cm}
<{\centering}p{2.8cm}<{\centering}p{0.8cm}<{\centering}}
\toprule
$n$ & $|\curl_h(\bm{u}-\bm{I}_{3h}\bm{u}_h)|_{1,h}$ & order & $\|\curl_h(\bm{u}-\bm{I}_{3h}\bm{u}_h)\|_0$  & order & $\|\bm{u}-\bm{I}_{3h}\bm{u}_h\|_0$ & order \\
\midrule
$6$      &6.294E1      &       &1.040E0      &         &1.582E-1    &         \\
$12$     &1.833E1      &1.78   &1.998E-1     &2.38     &2.329E-2    &2.76     \\
$18$     &8.420E0      &1.92   &8.965E-2     &1.98     &9.825E-3    &2.13     \\
$24$     &4.790E0      &1.96   &5.107E-2     &1.96     &5.534E-3    &2.00     \\
$36$     &2.146E0      &1.98   &2.298E-2     &1.97     &2.480E-3    &1.98     \\
$48$     &1.210E0      &1.99   &1.299E-2     &1.98     &1.401E-3    &1.98     \\
\bottomrule
\end{tabular}
\caption{Superconvergence of the modified scheme (\ref{e: modified discrete formulation}) using the postprocessing operator $\bm{I}_{3h}$.\label{t: superconvergence}}
\end{center}
\end{table}

Let us turn to the superclose and superconvergence results.
It can be easily verified from Tables \ref{t: supercloseness} and \ref{t: superconvergence}
that $\bm{I}_h\bm{u}$ is a superclose interpolation of order two, as claimed in Theorem \ref{th: supercloseness}, and $\bm{I}_{3h}$ is a proper postprocessing operator, as asserted in
Theorem \ref{th: superconvergence}. Let us proceed to check the efficiency of $\bm{I}_{3h}$ in a more detailed manner.
Indeed, although the errors of the postprocessed and unprocessed solutions measured in $|\curl_h\cdot|_{1,h}$ are almost equal over a coarser mesh ($n\approx10$),
when the mesh is refined sufficiently,
the postprocessed solutions have a significant higher precision than that in Table \ref{t: modified}.
For instance, when $n=48$, the error is almost equal to that when $n\approx200$ if no postprocessing is carried out. Furthermore, the performance seems more attractive if the errors are measured in $L^2$-norm (e.g.~when $n=48$, 2.643E-2 vs.~1.401E-3).
This indicates high efficiency of our method.


\begin{thebibliography}{99}
{\small
\addtolength{\itemsep}{-1.5ex}

\bibitem{Boffi2013}
D. Boffi, F. Brezzi, M. Fortin.
Mixed Finite Element Methods and Applications.
Springer, 2013.

\bibitem{Chen2002}
C.M. Chen. Structure theory of superconvergence of finite elements (in Chinese).
Hunan Science Press, 2002.

\bibitem{Hu2016}
J. Hu, R. Ma. Superconvergence of both the Crouzeix--Raviart and Morley elements.
Numer. Math., 2016, 132(3): 491--509.

\bibitem{Hu2016b}
J. Hu, Z. Shi, X. Yang. Superconvergence of three dimensional Morley
elements on cuboid meshes for biharmonic equations.
Adv. Comput. Math., 2016, 42(6): 1453--1471.

\bibitem{Hu2022}
K. Hu, Q. Zhang, Z. Zhang.
A family of finite element stokes complexes in three dimensions.
SIAM J. Numer. Anal., 60(1): 222--243, 2022.

\bibitem{Huang2023}
X. Huang.
Nonconforming finite element Stokes complexes in three dimensions. Sci. China Math., 2023: 1--24.

\bibitem{Lin1996}
Q. Lin, N. Yan.
The Construction and Analysis of High Efficiency Finite Element Methods (in
Chinese). Hebei University Publishers, Baoding, 1996.

\bibitem{Lin2000}
Q. Lin, N. Yan. Global superconvergence for Maxwell's equations.
Math. Comput., 2000, 69(229): 159--176.

\bibitem{Lin2005}
Q. Lin, L. Tobiska, A. Zhou. Superconvergence and extrapolation of non-conforming low order finite
elements applied to the Poisson equation. IMA. J. Numer. Anal., 2005, 25: 160--181.

\bibitem{Liu2008}
H. Liu, N. Yan. Superconvergence analysis of the nonconforming quadrilateral linear-constant scheme for Stokes equations. Adv. Comput. Math., 2008, 29(4): 375--392.

\bibitem{Linke2014}
A. Linke. On the role of the Helmholtz decomposition in mixed methods for incompressible flows and a new variational crime. Comput. Methods Appl. Mech. Engrg., 2014, 268: 782--800.

\bibitem{Mao2009}
S. Mao, Z. Shi. High accuracy analysis of two nonconforming plate elements.
Numer. Math., 2009, 111: 407--443.

\bibitem{Monk2003}
P. Monk. Finite Element Methods for Maxwell's Equations. Oxford University Press, 2003.

\bibitem{Nedelec1980}
J.C. Nédélec. Mixed finite elements in $\mathbb{R}^3$. Numer. Math., 1980, 35: 315--341.

\bibitem{Wang2021}
L. Wang, H. Li, Z. Zhang. $H(\curl^2)$-conforming spectral element method for quad-curl problems. Comput. Methods Appl. Math., 21(3): 661--681, 2021.

\bibitem{Wang2023}
L. Wang, M. Zhang, Q. Zhang.
Fully $H(\mathrm{grad}\,\mathrm{curl})$-nonconforming finite element method for the singularly perturbed quad-curl problem on cubical meshes. arXiv preprint, arXiv: 2301.03172, 2023.

\bibitem{Wang2023b}
L. Wang, Q. Zhang, Z. Zhang.
Superconvergence Analysis of $\curl\curl$-Conforming Elements on Rectangular Meshes.
J Sci. Comput., 2023, 95(2): 62.

\bibitem{Ye2002}
X. Ye. Superconvergence of nonconforming finite element method for the Stokes equations.
Numer. Methods Part. Differ. Equ. 2002, 18: 143--154.

\bibitem{Zhang2023}
B. Zhang, Z. Zhang. Polynomial preserving recovery and a posteriori error estimates for the two-dimensional quad-curl problem. Discrete Cont. Dyn. B, 2023, 28(2): 1323--1341.

\bibitem{Zhang2009}
S. Zhang, X. Xie, Y. Chen. Low order nonconforming rectangular finite element methods for Darcy-
Stokes problem. J. Comput. Math., 2009, 27: 400–424.

\bibitem{Zhang2020}
Q. Zhang, Z. Zhang.
A family of curl-curl conforming finite elements on tetrahedral meshes. CSIAM Trans. Appl. Math., 1(4): 639--663, 2020.

\bibitem{Zhang2022}
Q. Zhang, M. Zhang, Z. Zhang.
Nonconforming finite elements for the Brinkman and $-\curl\Delta\curl$ problems
on cubical meshes. arXiv preprint, arXiv: 2206.08493v2, 2022.

\bibitem{Zheng2011}
B. Zheng, Q. Hu, J. Xu.
A nonconforming finite element method for fourth order curl equations in $\mathbb{R}^3$.
Math. Comput., 80(276): 1871--1886, 2011.

\bibitem{Zhou2023}
X. Zhou, H. Niu, Z. Meng, J. Su.
Superconvergence of some nonconforming brick elements for the 3D Stokes problem.
under review.
}
\end{thebibliography}
\end{document}